\documentclass[runningheads,a4paper]{llncs}

\usepackage{graphicx}

\usepackage{amssymb,amscd,verbatim, latexsym}
\usepackage{enumerate}
\usepackage{multicol}
\usepackage[mathscr]{eucal}
\usepackage{epsfig,layout,graphics}
\usepackage{makeidx}

\usepackage{url}
\urldef{\mailsa}\path|hli@math.wayne.edu, |
\urldef{\mailsb}\path| nistor@math.psu.edu,|
\urldef{\mailsc}\path|yqiao@snnu.edu.cn|

\newcommand{\NN}{\mathbb N}

\newcommand{\RR}{\mathbb R}

\newcommand{\CI}{{\mathcal C}^{\infty}}

\newcommand\pa{{\partial}}

\newcommand{\maC}{\mathcal C}
\newcommand{\maD}{\mathcal D}

\newcommand{\maK}{\mathcal K}
\newcommand{\maL}{\mathcal L}

\newcommand{\maR}{\mathcal R}

\newcommand{\maV}{\mathcal V}
\newcommand{\maW}{\mathcal W}

\newcommand{\ie}{{\em i.\thinspace e.,\ }}
\newcommand\comm[1]{{}}

\newcommand{\mcb}{\maC_b}
\newcommand\famP{$(P_y)_{y \in U}$}

\begin{document}

\mainmatter  % start of an individual contribution

% first the title is needed
\title{Uniform shift estimates for transmission problems and optimal
  rates of convergence for the parametric Finite Element Method}

% a short form should be given in case it is too long for the running head
\titlerunning{Uniform Shift Estimates For Parametric Transmission Problems}

% the name(s) of the author(s) follow(s) next
%
% NB: Chinese authors should write their first names(s) in front of
% their surnames. This ensures that the names appear correctly in
% the running heads and the author index.
%
\author{Hengguang Li \inst{1} %
\and Victor Nistor \inst{2} \and Yu Qiao\inst{3}
\thanks{H. Li was partially supported by the NSF Grant DMS-1158839. V. Nistor was partially supported by the NSF Grant DMS-1016556.}}
\institute{Springer-Verlag, Computer Science Editorial,\\
Tiergartenstr. 17, 69121 Heidelberg, Germany\\
\mailsa\\}

\institute{Department of Mathematics, Wayne State University, Detroit, MI 48202, USA\\
\mailsa
\and
Department of Mathematics, Pennsylvania State University\\
University Park, PA, 16802, USA\\
Inst. Math. Romanian Acad. \\
PO BOX 1-764, 014700, Bucharest, Romania\\
\mailsb
\and
College of Mathematics and Information Science\\
Shaanxi Normal University, Xi'an, Shaanxi, 710062, P.R.China\\
\mailsc}

\authorrunning{Hengguang Li, Victor Nistor, and Yu Qiao}
% (feature abused for this document to repeat the title also on left hand pages)

% the affiliations are given next; don't give your e-mail address
% unless you accept that it will be published
%\institute{Springer-Verlag, Computer Science Editorial,\\
%Tiergartenstr. 17, 69121 Heidelberg, Germany\\
%\mailsa\\
%\mailsb\\
%\mailsc\\
%\url{http://www.springer.com/lncs}}

%
% NB: a more complex sample for affiliations and the mapping to the
% corresponding authors can be found in the file "llncs.dem"
% (search for the string "\mainmatter" where a contribution starts).
% "llncs.dem" accompanies the document class "llncs.cls".
%

%\toctitle{Lecture Notes in Computer Science}
%\tocauthor{Authors' Instructions}

\maketitle

\begin{abstract}
Let $\Omega \subset \RR^d$, $d \geqslant 1$, be a bounded domain with
piecewise smooth boundary $\partial \Omega $ and let $U$ be an open
subset of a Banach space $Y$. Motivated by questions in ``Uncertainty
Quantification,'' we consider a parametric family $P = (P_y)_{y \in
  U}$ of uniformly strongly elliptic, second order partial
differential operators $P_y$ on $\Omega$. We allow jump
discontinuities in the coefficients. We establish a regularity result
for the solution $u: \Omega \times U \to \RR$ of the parametric,
elliptic boundary value/transmission problem $P_y u_y = f_y$, $y \in
U$, with mixed Dirichlet-Neumann boundary conditions in the case when
the boundary and the interface are smooth and in the general case for
$d=2$. Our regularity and well-posedness results are formulated in a
scale of broken weighted Sobolev spaces $\hat
\maK^{m+1}_{a+1}(\Omega)$ of Babu\v{s}ka-Kondrat'ev type in $\Omega$,
possibly augmented by some locally constant functions. This implies
that the parametric, elliptic PDEs $(P_y)_{y \in U}$ admit a shift
theorem that is uniform in the parameter $y\in U$. In turn, this then
leads to $h^m$-quasi-optimal rates of convergence (\ie algebraic
orders of convergence) for the Galerkin approximations of the solution
$u$, where the approximation spaces are defined using the ``polynomial
chaos expansion'' of $u$ with respect to a suitable family of
tensorized Lagrange polynomials, following the method developed
by Cohen, Devore, and Schwab (2010).\end{abstract}

\section{Introduction}

Recently, questions related to differential equations with random
coefficients have received a lot of attention due to the practical
applications of these problems
\cite{BabuskaTempone2,Schwab2,Karniadakis1,SchwabTodor1,Karniadakis2}. Our
paper is motivated by the approach in \cite{Schwab1,NS1}, where
families of differential operators on polyhedral domains indexed by $y
\in U$, were studied.  As in those papers, $U$ is an open subset of a
Banach space $Y$, which allows us to study the analyticity of the
solution in terms of $y \in U$.

We here study the properties of solutions to a family of strongly
elliptic, mixed boundary value/transmission problems
\begin{equation}\label{eq.One}
  P_y u_y(x) = Pu(x,y) = f_y(x) = f(x,y), \quad x \in \Omega, \ y
  \in U
\end{equation}
on a domain $\Omega \subset \RR^d$, $d \ge 1$. The domain $\Omega $ is
assumed to be piecewise smooth and bounded. Thus, for each $y \in U$,
we are given a second order, uniformly strongly positive, parametric
partial differential operator $P_y$ on $\Omega $ whose coefficients
are functions of $(x,y) \in \Omega \times U$ and are allowed to have
jump discontinuities across a fixed interface $\Gamma$. More
precisely, we assume that $\overline{\Omega} = \cup_{k=1}^K
\overline{\Omega}_k$, where $\Omega_k$ are disjoint domains with
piecewise smooth boundaries and $\Gamma := \big( \cup_{k=1}^K \pa
\Omega_k \big) \smallsetminus \pa \Omega$.

Under suitable regularity assumptions on the coefficients of $P$ and
on the source term $f: \Omega \times U \to \RR$, we establish in
Section \ref{sec.fam.reg} a regularity and well-posedness result for
the solution $u: \Omega \times U \to \RR$ of the parametric, elliptic
boundary value/transmission problem (\ref{eq.One}) with mixed
Dirichlet-Neumann boundary conditions. Our regularity result is
formulated in a scale of broken weighted Sobolev spaces $\hat
\maK^{m+1}_{a+1}(\Omega) = \oplus_{k=1}^K \maK_{a+1}^{m+1}(\Omega_k)$
of Babu\v{s}ka-Kondrat'ev type in $\Omega$, for which we prove that
our elliptic PDEs $(P_y)_{y \in U}$ admit a shift theorem that is
uniform in the parameter $y \in U$. We deal completely in this paper
with the cases when the boundary $\pa \Omega$ and the interface
$\Gamma$ are smooth and disjoint. We also indicate how to proceed in
the general case for $d=2$. Our results generalize the results of
\cite{NS1} by allowing jump discontinuities in the coefficients and by
allowing adjacent edges to be endowed with Neumann-Neumann boundary
conditions. We will be therefore brief in our presentation, referring
to \cite{NS1}, as well as \cite{Schwab1,Schwab2} for more details.

The main contribution of this paper is to study the regularity of the
solution of a (non-parametric) transmission/boundary value problem
with rather weak smoothness assumptions on the coefficients. As far as
we know, this paper is the only place where a complete proof for the
regularity of transmission problems is given, even in the case of
smooth coefficients. The results are general enough so that one can
use the approach in \cite{Schwab1,NS1} to obtain regularity results
for families and then to obtain optimal rates of convergence for the
Galerkin method. An abstract version of this method is explained in
\cite{BacutaNistorParam}. These issues will be discussed in more
detail in a forthcoming paper.

The paper is organized as follows. In Section \ref{sec.first} we
formulate our parametric partial differential boundary
value/transmission problem and introduce some of our main
assumptions. We also discuss the needed notions of positivity for
families of operators and derive some simple consequences. In Section
\ref{sec.two}, we review the ``broken" version of usual Sobolev
spaces, and then formulate and prove the main results, Theorem
\ref{thm.smooth}, which is a regularity and well-posedness result for
non-parametric solutions in smooth case.  In Section \ref{sec.three},
we recapitulate regularity and well-posedness results for the
non-parametric, elliptic problem from \cite{BNZ1,BNZ3D2,HMN}, the main
result being Theorem \ref{thm.one.op}. This theorem is then
generalized to families in Section \ref{sec.fam.reg}, thus yielding
our main regularity and well-posendess result for parametric families
of uniformly strongly boundary value/transmission problems, namely
Theorem \ref{thm.fam.op}. As mentioned above, this result is
formulated in broken weighted Sobolev spaces (the so called
``Babu\v{s}ka-Kondrat'ev'' spaces).  As in \cite{Schwab1,Schwab2,NS1}
these results lead to $h^m$-quasi-optimal rates of convergence for a
suitable Galerkin method for the approximation of our parametric
solution $u$.

\emph{Acknowledgements.}\ VN acknowledges the support of the Hausdorff
Institute for Mathematics (HIM) in Bonn during the HIM Trimester
``High dimensional approximation,'' where this work was initiated.  We
also thank Christoph Schwab for several useful discussions.

\section{Ellipticity, positivity, solvability for parametric families}
\label{sec.first}
We now formulate our parametric partial differential boundary
value/transmission problem and introduce some of our main assumptions.

\subsection{Notation and assumptions\label{ssec:Notation}}
By $\Omega \subset \RR^d$, $d \ge 1$, we shall denote a connected,
bounded piecewise smooth domain, which we assume is decomposed into
finitely many subdomains $\Omega_k$ with piecewise smooth boundary,
$\overline{\Omega} = \cup_{k=1}^K \overline{\Omega}_k$. We obtain
results on the spatial regularity of PDEs whose data depend on a
parameter vector $y \in U \subset Y$, where $U$ is an open subset of a
Banach space $Y$. By $a^{ij}_{pq}, b^{i}_{pq}, c_{pq} : \Omega \times
U \to \RR$, $1 \le i,j \le d$, we shall denote bounded, measurable
functions satisfying smoothness and other assumptions to be made
precise later. We denote by $A = (a^{ij}_{pq}, b^{i}_{pq},
c_{pq})$. Let us denote by $\pa_i = \frac{\pa}{\pa x_i}$, $i=1,
\ldots, d$. We shall then denote by $P^A = [P^A_{pq}]$ a $\mu \times
\mu$ matrix of parametric differential operators in divergence form
\begin{eqnarray}\label{eq.action}
  P_{pq}^A u(x, y) := \left(- \displaystyle{\sum_{i,j=1}^d} \pa_i \big(
  a^{ij}_{pq}(x,y) \pa_j \big)
  + \displaystyle{\sum_{i=1}^d} b^i_{pq}(x,y) \pa_i  +
  c_{pq}(x, y)\right) u(x, y),
\end{eqnarray}
where $x \in \Omega $ and $y \in U$. Note that the derivatives act
only in the $x$-direction and $y$ is just a parameter. The matrix case
is needed in order to handle the case of systems, such as that of
(anisotropic) linear elasticity.

A matrix $P = [P_{pq}]_{p,q=1}^\mu$ of differential operators acts on
vector-valued functions $u=(u_q)_{q=1}^\mu$ in the usual way $(Pu)_p =
\sum_{q=1}^{\mu} P_{pq} u_q,$ for $u = (u_q) \in \CI(\Omega \times
U)^\mu.$ We recall that $H^{-1}(\Omega )$ is defined as the dual of
$H^1_0(\Omega ) := \{u \in H^1(\Omega ), u\vert_{\pa \Omega } = 0\}$
with pivot $L^2(\Omega )$. Occasionally, we shall need to specialize a
family $P$ for a particular value of $y$, in which case we shall write
$P_y : \CI(\Omega )^\mu \to H^{-1}(\Omega )^\mu$ for the induced
operator. We emphasize that we allow $P$ to have non-smooth
coefficients, so that $P u$ may be non-smooth in general.

\subsection{Boundary and interface conditions}
\label{ssec:BounCond}
We impose mixed Dirichlet and Neumann boundary conditions. To this
end, we assume there is given a closed set $\pa_D \Omega \subset \pa
\Omega$, which is a union of polygonal subsets of the boundary and we
let $\pa_N \Omega := \pa \Omega \smallsetminus \pa_D \Omega$. The set
$\pa_D \Omega$ will be referred to as ``Dirichlet boundary'' and
$\pa_N \Omega$ as ``Neumann boundary,'' according to the type of
boundary conditions that we associate to these parts of the
boundary. The case of cracks is also allowed, provided that one treats
different sides of the crack as {\em different} parts of the boundary,
as in \cite{HMN}, for instance, but we choose not to treat this case
explicitly in this paper. We then define the {\em conormal
  derivatives}
\begin{equation}\label{eq.conormal}
  (\nabla^{A}_{\nu} u)_{p}(x, y) = \sum_{q=1}^{\mu} \sum_{i,j=1}^{d}
  \nu_i a^{ij}_{pq}(x, y) \pa_j u_q(x, y), \quad x \in \pa_N \Omega
  ,\;y\in U,
\end{equation}
where $\nu = (\nu_i)$ is the outward unit normal vector at $x\in \pa_N
\Omega $. The conormal derivatives $\nabla^{A}_{\nu} u^{\pm}$ at
  the interface $\Gamma$ are defined similarly, using an arbitrary but
  fixed labeling of the two sides of the interface into a positive and
  a negative part.

We shall also need the spaces $H^1_{D}(\Omega )$ and
$H^{-1}_{D}(\Omega )$ for vector-valued functions: $H_{D}^{1}(\Omega)
:= \{ u \in H^{1}(\Omega )^\mu ,\, u = 0\, \mbox{ on }\, \pa_D \Omega
\}$ and $H^{-1}_{D}(\Omega )$, defined to be the dual of
$H^1_{D}(\Omega )$ with pivot space $L^2(\Omega )$. Note that we
assume here as in \cite{NS1} that we have {\em the same type of
  boundary conditions for all components $u_q$ of the solution vector
  $u$.}

Recall that our domain $\Omega$ is decomposed into $K$ subdomains of
the same type (with piecewise smooth boundary), $\overline{\Omega} =
\cup_{k=1}^K \overline{\Omega}_k$. We then denote by $\Gamma := \big (
\cup_{k=1}^K \pa \Omega_j \big) \smallsetminus \pa \Omega$ the
interface of our problem. We also fix arbitrarily the sides of the
interface, and we thus denote by $u^{+}$, respectively $u^{-}$ the
non-tangential limits of $u$ at the two sides of the interface. We
define similarly the conormal derivatives $\nabla^{A}_{\nu}u^{+}$ and
$\nabla^{A}_{\nu}u^{-}$ at the interface, but using the two sided
limits of the coefficients $a_{pq}^{ij}$ at $\Gamma$. We consider the
parametric family of boundary value/interface problems
\begin{equation}\label{eq.pBVP}
\left\{\begin{array}{ll}
   \ \ P^A u(x, y) = f(x, y) & \quad x \in \Omega , \\
   \ \  u(x,y) = 0 & \quad x \in \pa_D \Omega ,\\
   \ \ \nabla^A_{\nu} u(x,y) = g(x,y) & \quad x \in \pa_N \Omega \\
   \ \ u^{+}(x, y) - u^{-}(x, y) = 0 & \quad x \in \Gamma \\
   \ \ \nabla^{A}u^{+}(x, y) - \nabla^{A} u^{-}(x, y) = h(x, y) &
   \quad x \in \Gamma \\
\end{array}\right.
\end{equation}
where $P^A$ is as in Equation (\ref{eq.action}), $\nabla^{A}_{\nu}$ is
as in Equation (\ref{eq.conormal}), and $y \in U$. We stress
  that for us the dependence of $P^A$ on its coefficients, that is on
  $A$, is important, which justifies our notation.

\subsection{Ellipticity and positivity for differential 
operators\label{ssec:DOPs}}
In this subsection we recall the definition of the positivity property
for parametric families of differential operators.  Let us therefore
consider, for any $y \in U$, the parametric bilinear form $B(y;\,
\cdot \, ,\, \cdot \, )$ defined by
\begin{eqnarray}
  &B(y;v,w) \, := \, \int_{x\in \Omega } \, \sum_{p, q = 1}^\mu \Big(
  \sum_{i,j=1}^{d} a^{ij}_{pq}(x,y) \partial_i v_p(x, y) \partial_j
  w_q(x, y)  \nonumber\\
  &\sum_{i=1}^{d} b^{i}_{pq}(x,y) \partial_i v_p(x, y) w_q(x, y) +
  c_{pq}(x,y) v_p(x, y) w_q(x, y) \Big) dx  \, , \quad y\in U. \label{eq:defby}
\end{eqnarray}

\begin{definition}\label{def.u.positive}
The family $(P_y)_{y \in U}$ is called {\em uniformly strictly
  positive definite} on $H^1_0(\Omega)^\mu \subset V \subset
H^1(\Omega )^\mu$ if the coefficients $a^{ij}_{pq}$ are symmetric in
$i,j$ and in $p,q$ (that is, $a^{ij}_{pq} = a^{ji}_{pq} =
a^{ij}_{qp}$, for all $i$, $j$, $p$, and $q$), and if there exist $0 <
r < R < \infty$ such that for all $y \in U$, and $v,w \in V$, we have
\begin{eqnarray*}
  |B(y; v, w)| \le R \| v\|_{H^{1}(\Omega )} \| w\|_{H^{1}(\Omega )}
  \quad {\mbox{and}} \quad 
  r \|v\|_{H^{1}(\Omega )}^2 \, \le \, B(y; v, v).
\end{eqnarray*}
\end{definition}

If $U$ is reduced to a single point, that is, if we deal with the case
of a single operator instead of a family, then we say that $P$ is
strictly positive definite. Throughout this paper, we shall assume
that {\em $(P_y)_{y \in U}$ is uniformly strictly positive
  definite.} Positivity is closely related to ellipticity.

\begin{definition}\label{def.u.s.elliptic}
The family $(P_y)_{y \in U}$ is called {\em uniformly strongly
elliptic} if the coefficients $a^{ij}_{pq}$ are symmetric in $i,j$
and in $p,q$ and if there exist $0 < r_e < R_e < \infty$ such that
for all $x \in D$, $y \in U$, $\xi \in \RR^d$, and $\eta \in \RR^\mu$
\begin{eqnarray}\label{eq.ineq.se}
  r_e |\xi|^2|\eta|^2 \, \le \, \sum_{p, q=1}^\mu \sum_{i,j=1}^d
  a^{ij}_{pq}(x, y) \xi_i \xi_j \eta_p \eta_q
  \le \, R_e |\xi|^2|\eta|^2 \;.
\end{eqnarray}
\end{definition}

In case one is interested only in scalar equations (not in systems),
then for $V = H^1_{D}(\Omega)$, the assumption that our family $P_y$
is uniformly positive definite can be replaced with the (slightly
weaker) assumption that the family $P_y$ is uniformly strongly
elliptic, that $\sum_{i=1}^d \pa_i b^i = 0$ in $\Omega$, $\sum_{i=1}^d
\nu_i b^i = 0$ on $\pa_N \Omega$, $c \ge 0$, and $\pa_D \Omega \neq
\emptyset$ (in which case it also follows that $\pa_D \Omega$ has a
non-empty measure). In general, a uniformly strictly positive family
$P$ will also be uniformly strongly elliptic.

\subsection{Consequences of positivity}
\label{ssec:ConseqPos}
The usual Lax-Milgram lemma gives the following result as in
\cite{Schwab1}. Recall the constant $r$ from Definition
\ref{def.u.positive}.

\begin{proposition} \label{prop.existence}
Assume that $f_y := f(\,\cdot \, , y) \in H_{D}^{-1}(\Omega)$, for any
$y \in U$. Also, assume that the family $P$ is uniformly strictly
positive definite. Then our family of boundary value problems $P_y u_y
= f_y$, $u_y \in H^1_{D}(\Omega )$, \ie Equation (\ref{eq.pBVP}),
admits a unique solution $u_y = P_y^{-1} f_y$. Moreover,
$\|P_y^{-1}\|_{\maL(H_{D}^{-1}; H_{D}^1)} \le r^{-1},$ for all $y\in
U.$
\end{proposition}

The parametric solution $u_y \in H^1_{D}(\Omega )$ of Proposition
\ref{prop.existence} is then obtained from the usual weak
formulation:\ given $y \in U$, find $u_y \in V := H^1_{D}(\Omega )$
such that
\begin{equation}\label{eq.weak}
  B(y; u_y, w) = (f_y,w) + \int_{\pa_N \Omega} g_yw dS + \int_{\Gamma}
  h_yw dS , \quad \forall w \in V,
\end{equation}
where $(f_y, w)$ denotes the $L^2(\Omega)$ inner product and 
$dS$ is the surface measure on $\pa \Omega$ or on $\Gamma$. Also, $f_y(x) = f(x,y)$,
and similarly for $u_y$, $g_y$, and $h_y$.

\section{Broken Sobolev spaces and higher regularity of 
non-parametric solutions in the smooth case\label{sec.two}}

One of our main goals is to obtain regularity of the solution $u$ both
in the space variable $x$ and in the parameter $y$. It is convenient
to split this problem into two parts: regularity in $x$ and regularity
in $y$. We first address regularity in $x$ in the case when the
boundary $\pa \Omega$ and the interface $\Gamma$ are smooth and
disjoint.  We also assume that each connected component of the
boundary is given a single type of boundary conditions: either
Dirichlet or Neumann. This leads to Theorem \ref{thm.smooth}, which
states the regularity and well-posedness of Problem (\ref{eq.pBVP}) in
this smooth case ($\pa \Omega$ and $\Gamma$ smooth and disjoint). This
is the main result of this paper, and, as far as we know, no complete
proof was given before. We also consider coefficients with lower
regularity than it is usually assumed, which is needed to treat the
truly parametric case. (We are planning to address this question in a
future paper.)

{\em We assume throughout this and the following section that we are
  dealing with a single, non-parametric equation (not with a family),
  that is, that $U$ is reduced to a single point in this
  subsection. We also assume that $\mu =1$, to simplify the notation.}

We shall need the ``broken'' version of the usual Sobolev spaces to
deal with our interface problem. Recall the subdomains
$\Omega_k\subset \Omega$, $1\leq k\leq K$, we  define:
\begin{eqnarray}
 & \hat H^{m} (\Omega ) := \{ v: \Omega \to \RR, v \in H^m(\Omega_k),
  \ \forall 1 \leq k \leq K \} \label{eq.def.broken} \\
  & \hat W^{m, \infty} (\Omega ) := \{ v: \Omega \to \RR, \, \pa^\alpha
  v \in L^{\infty}(\Omega_k),\ \forall 1 \leq k \leq K, |\alpha| \le m
  \}.  \nonumber
\end{eqnarray}
For further reference we note that the definitions of these spaces
imply that the multiplication and differentiation maps $\hat W^{m,
  \infty} (\Omega ) \times \hat H^{m} (\Omega ) \to \hat H^{m} (\Omega
)$ and $\quad \pa_i : \hat H^{m} (\Omega ) \to \hat H^{m-1}(\Omega )$
are continuous.

One of the difficulties of dealing with interface problems is the more
complicated structure of the domains and ranges of our operators.
When $m=0$, we define $\maD_{m} = \maD_{0} = H^1_D(\Omega) = V$ and
$\maR_{m} = \maR_{0} = H^{-1}_D(\Omega) = V^*.$ Then we define $\tilde
P^{A}$ in a weak sense using the bilinear form $B$ introduced in
Equation (\ref{eq:defby}) (see the discussion around Equation (2.12)
in \cite{HMN} for more details or the discussion around Equation (20)
in \cite{MN2}). Assume now that $m \ge 1$.
We then define 
\begin{eqnarray*}
%\begin{gathered}
 & \maD_{m} := \hat H^{m+1}(\Omega ) \cap \{u=0
   {\mbox{ on }} \pa_D \Omega \} \cap  \{u^+ - u^{-} = 0 \mbox{ on
    } \Gamma \}  \ \mbox{ and } \\
 & \maR_{m} := \hat H^{m-1}(\Omega) \oplus H^{m-1/2}(\pa_N \Omega)
  \oplus H^{m-1/2}(\Gamma).
%\end{gathered}
\end{eqnarray*}
In particular, $\maD_{m} = \hat H^{m+1}(\Omega ) \cap
H^1_D(\Omega)$. Let $A = (a^{ij}, b^{i}, c) \in \hat
W^{m,\infty}(\Omega)^{d^2 + d + 1}$ and $P^{A} u = \sum_{i,j=1}^{2}
\pa_i(a^{ij} \pa_j u) + \sum_{i=1}^{2} b^{i} \pa_i u + c u$, as
before. Then the family of partial differential operators $\tilde
P^A_{m} : \maD_{m} \to \maR_{m},$
\begin{equation}
%\begin{gathered}
%
  \tilde P^A_{m} u = \big (P^A u, \nabla^A_{\nu} u\vert_{\pa_N
    \Omega}, (\nabla^A_{\nu} u^{+}- \nabla^A_{\nu} u^{+})
    \vert_{\Gamma} \big)
% \end{gathered} 
\end{equation}
is well defined. Note that the domain $\maD_{m}$ and codomain
$\maR_{m}$ are independent of $y \in U$, which justifies why we do not
consider homogeneous Neumann boundary conditions. We are now ready to
state and prove our main theorem. Let us denote $\|u\|_{\hat
  H^m(\Omega )} := \Big (\sum_{k=1}^K \|u\|_{H^m(\Omega_k)^2}
\Big)^{1/2}$ and $\|v\|_{\hat W^{m, \infty}(\Omega)} := \sum_{k=1}^K
\|v\|_{W^{m, \infty}(\Omega_k)}$ the resulting natural norms on the
spaces introduced in Equation (\ref{eq.def.broken}).

\begin{theorem}\label{thm.smooth} 
Let us assume that $\Omega \subset \RR^d$ is smooth and bounded, that
the interface $\Gamma$ is smooth and does not intersect the boundary,
and that to each component of the boundary it is associated a single
type of boundary conditions (either Dirichlet or Neumann).  Assume
that $A = (a^{ij}, b^i, c) \in \hat W^{m, \infty}(\Omega)^{d^2 + d
  +1}$ and that $P^A$ is strictly positive definite on $H^1_{D}(\Omega
)$, then $\tilde P^A_{m} : \maD_{m} \to \maR_{m}$ is invertible.

Moreover, let $\|P^{-1}\|$ denote norm of the
inverse of the map $P : H_{D}^1(\Omega ) \to H_{D}^1(\Omega )^* =:
H_{D}^{-1}(\Omega )$. Then there exists a constant $\tilde C_1 > 0$
such that the solution $u$ of (\ref{eq.pBVP}) satisfies
\begin{equation}
 \|u\|_{\hat H^{m+1}(\Omega )} + \|u \|_{H^{1}(\Omega )} \le
 \tilde C_1 \big ( \|f\|_{\hat H^{m-1}(\Omega )} \nonumber 
 \|g\|_{H^{m-\frac{1}{2}}(\pa_N \Omega)} +
 \|h\|_{H^{m-\frac{1}{2}}(\Gamma)} \big),\label{eq.lemma.D.smooth}
\end{equation}
with the constant $\tilde C_1=\tilde C_1(m, \|P^{-1}\|,
\|A\|_{\hat W^{m,\infty}})$.% depending only on the indicated variables.
\end{theorem}

\begin{proof} 
In the case of the pure Dirichlet boundary conditions for an equation
and without the explicit bounds in Equation (\ref{eq.lemma.D.smooth}),
this lemma is a classical result, which is proved using divided
differences and the so called ``Nirenberg's trick'' (see
\cite{Fichera,Morrey}). Since we consider transmission problems and
want the more explicit bounds in the above Equation
(\ref{eq.lemma.D.smooth}), let us now indicate the main steps to treat
the interface regularity following the classical proof and
\cite{NS1}. The boundary conditions (\ie regularity at the boundary)
were dealt with in \cite{NS1}. In all the calculations below, all the
constants $C$ in this proof will be generic constants that will depend
only on the variables on which $\tilde C_1$ depends (\ie on the order
$m$, the norms $\|P^{-1}\|$ and $\|A\|_{\hat W^{m,\infty}(\Omega)}$).
We split the proof into several steps.

{\em Step 1.}\ We first use Proposition \ref{prop.existence} to
conclude that $P : H_{D}^1(\Omega) \to H_{D}^1(\Omega )^*$ is indeed
invertible.  This provides the needed estimate for $m = 0$ (in which
case, we recall, our problem (\ref{eq.pBVP}) has to be interpreted in
a weak sense).

{\em Step 2.}\ For $m>0$ we can assume $g = 0$ and $h = 0$ by using
the extension theorem as in \cite{NS1}.

{\em Step 3.}\ We also notice that, in view of the invertibility of
$P$ for $m = 0$ and since $u \in H^1_D(\Omega)$, it suffices to prove
\begin{equation}\label{eq.p.simple}
  \|u\|_{\hat H^{m+1}(\Omega)} \, \le \, C \big( \sum_k
  \|f\|_{H^{m-1}(\Omega_k)} + \|u \|_{\hat H^{m}(\Omega )} \big) \;.
\end{equation}
Indeed, the desired inequality (\ref{eq.lemma.D.smooth}) will follow
from Equation (\ref{eq.p.simple}) by induction on $m$. Since Equation
(\ref{eq.p.simple}) holds for $P$ if, and only if, it holds for
$\lambda + P$, in order to prove Equation (\ref{eq.p.simple}), 
it is also enough to assume that $\lambda + P$ is strictly positive 
for some $\lambda \in \RR$. 
In particular, Equation
(\ref{eq.p.simple}) will continue to hold--with possibly different
constants--if we change the lower order terms of $P$.

{\em Step 4.}\ Let us assume that $\Omega = \RR^d$ with the interface
given by $\Gamma = \{x_d = 0\}$. Let $\Omega_1 = \RR^d_+$ and
$\Omega_2 = \RR^d_{-}$ be the two halves into which $\RR^d$ is divided
(so $K = 2$). Then we prove Equation (\ref{eq.p.simple}) for these
particular domains and for $g = 0$ and $h = 0$ by induction on $m$. As
we have noticed, the Equation (\ref{eq.p.simple}) is true for $m=0$,
since the stronger relation (\ref{eq.lemma.D.smooth}) is true in this
case. Thus, we shall assume that Equation (\ref{eq.p.simple}) has been
proved for $m$ and for smaller values and we will prove it for $m+1$.
That is, we want to prove
\begin{equation}\label{eq.p.simple'}
  \|u\|_{\hat H^{m+2}(\RR^d)} \, \le \, C \big ( \|f\|_{\hat
    H^{m}(\RR^d)} + \|u\|_{\hat H^{m+1}(\RR^d)} \big ).
\end{equation} 
To this end, let us first write
\begin{equation}\label{eq.p.0}
  \|u\|_{\hat H^{m+2}(\RR^d)} \le \sum_{j=1}^d \|\pa_j u\|_{\hat
    H^{m+1}(\RR^d )} + \|u\|_{L^2(\RR^d)} \;.
\end{equation}
We then use our estimate (\ref{eq.p.simple}) for $m$ (using the
induction hypothesis) applied to the function $\pa_j u$ for $j <
d$. This gives
\begin{eqnarray}
&  \| \pa_j u\|_{\hat H^{m+1}(\RR^d )} \le \|P \pa_j u\|_{\hat
    H^{m-1}(\RR^d)} \le \|\pa_j f\|_{\hat H^{m-1}(\RR^d)} + \|[P, \pa_j] u\|_{\hat
    H^{m-1}(\RR^d)}\nonumber\\
 &  \le \|f\|_{\hat H^{m}(\RR^d)} + C \|u\|_{\hat
    H^{m+1}(\RR^d)}\label{eq.p.1}
\end{eqnarray} 
since the commutator $[P, \pa_j] = P\pa_j - \pa_j P$ is an operator of
order $\le 2$ whose coefficients can be bounded in terms of
$\|A\|_{\hat W^{m,\infty}(\Omega )}$. We now only need to estimate $\|\pa_d
u\|_{H^{m+1}}$, we do that on each half subspace.
\begin{eqnarray}
&  \|\pa_d u\|_{\hat H^{m+1}(\RR^d)} \le \sum_{j=1}^d \|\pa_j \pa_d
  u\|_{\hat H^{m}(\RR^d)} + \|\pa_d u\|_{L^2(\RR^d)} \nonumber \\
 & \le \sum_{j=1}^{d-1} \|\pa_j u\|_{\hat H^{m+1}(\RR^d)} + \|\pa_d^2
  u\|_{\hat H^{m}(\RR^d)} + \| u\|_{\hat H^1(\RR^d)} .\label{eq.p.2}
\end{eqnarray}
The right hand side of the above equation contains only terms that
have already been estimated in the desired way, except for $\|\pa_d^2
u\|_{H^{m}}$.  Since $m \ge 0$, we can use the relation $P u =f$ to
estimate this term as follows. Let us write $P u = \sum \pa_i (a^{ij}
\pa_j u) + \sum b^i \pa_i u + cu$. This gives $ a^{dd} \pa_d^2 u = f -
\sum_{(i, j) \neq (d, d)} a^{ij} \pa_i \pa_j u + Qu,$ where $Q$ is a
first order differential operator. Next we notice that $a^{dd}$ is
uniformly bounded from below by the uniform strong positivity property
(which implies uniform strong ellipticity): $(a^{dd})^{-1} \le
r^{-1}$. Note that by Proposition \ref{prop.existence}, we have
$\|P^{-1}\| \le r$, and hence $r^{-1}$ is an admissible constant. This
gives $\pa_d^2 u = (a^{dd})^{-1} f - \sum_{j=1}^{d-1} B^j \pa_j u +
Q_1u$ where $B^j$ and $Q_1$ are first order differential operators
with coefficients bounded by admissible constants, which then gives
\begin{eqnarray}\label{eq.p.4} 
  \|\pa_d^2 u\|_{\hat H^{m}(\RR^d)} \le C \big (\| f\|_{\hat
    H^{m}(\RR^d)} + \sum_{j=1}^{d-1} \|\pa_j u\|_{\hat H^{m+1}(\RR^d)}
  + \|u\|_{\hat H^{m+1}(\RR^d)} \big )\\
  \le C \big ( \| f\|_{\hat H^{m}(\RR^d)} + \|u\|_{\hat
    H^{m+1}(\RR^d)} \big ) \nonumber
\end{eqnarray} 
by Equation (\ref{eq.p.1}). Equation (\ref{eq.p.2}) and (\ref{eq.p.4})
then give
\begin{equation}\label{eq.p.5}
  \| \pa_d u\|_{\hat H^{m+1}(\RR^d)} \le C \big ( \|f\|_{\hat
    H^{m}(\RR^d)} + \|u\|_{\hat H^{m+1}(\RR^d)} \big )\, .
\end{equation} 
Combining Equations (\ref{eq.p.5}) and (\ref{eq.p.1}) with Equation
(\ref{eq.p.0}) gives then the desired Equation (\ref{eq.p.simple'})
for  $m$ replaced with $m+1$.

{\em Step 5.}\ We finally reduce to the case of a half-space or a full
space using a partition of unity as in the classical case, as
follows. We choose a smooth partition of unity $(\phi_j)$ on $\Omega $
consisting of functions with small supports. The supports should be
small enough so that if the support of $\phi_j$ intersects the
boundary of $\Omega$ or the interface $\Gamma$, then the boundary or
the interface can be straightened in a small neighborhood of the
support of $\phi_j$. We arrange that the resulting operators are
positive and we complete the proof as in \cite{NS1}.
\end{proof}

See also \cite{RoitbergSh62,RoitbergSh63}.

\section{Weighted Sobolev spaces and higher regularity of 
non-parametric solutions\label{sec.three}}

We now assume $d=2$, so $\Omega$ is a plane domain. We allow however
$\Omega$ to be {\em piecewise} smooth. We also consider coefficients
with lower regularity than the ones considered in \cite{HMN}. This
leads to Theorem \ref{thm.one.op}, which will be then generalized to
families in a forthcoming paper, which will contain also full details
for the remaining results. {\em We continue to assume that we are
  dealing with a single, non-parametric equation and that $\mu =1$.}

To formulate further assumptions on our problem and to state our
results, we shall need weighted Sobolev spaces, both of $L^2$ and of
$L^\infty$ type. Let $\maV$ be the set of singular points, where $Q
\in \maV$ if one of the following is satisfied:\ (1) it is a vertex,
(2) it is a point where the boundary condition changes type (from
Dirichlet to Neumann), (3) it is a point where the interface meets the
boundary, or (4) it is a non-smooth point on the interface of the
subdomains $\Omega_k$. Let us denote from now on by $\rho : \RR^2 \to
[0, 1]$ a continuous function that is smooth outside the set $\maV$
and is such that $\rho(x)$ is equal to the distance from $x \in \RR^2$
to $\maV$ when $x$ is close to the singular set $\maV$. The function
$\rho$ will be called the {\em smoothed distance to the set of
  singular points.} We can also assume $\| \nabla \rho \| \le 1,$
which will be convenient in later estimates, since it will reduce the
number of constants (or parameters) in our estimates. We first define
the {\em Babu\v{s}ka-Kondrat'ev spaces}
\begin{eqnarray} \label{eq.def.wbroken}
%\begin{gathered}
    \maK^{m}_{a} (\Omega) := \{ v : \Omega \to \RR, \, \rho^{|\alpha|
      - a} \pa^{\alpha} v \in L^2(\Omega ),\; \forall \;
    |\alpha|\leq m \} \\
    \maW^{m, \infty} (\Omega ) := \{ v : \Omega \to \RR, \,
    \rho^{|\alpha|} \pa^{\alpha} v \in L^\infty(\Omega ),\; \forall \;
    |\alpha|\leq m \}.
%\end{gathered}
\end{eqnarray}
We shall denote by $\| \; \cdot \; \|_{\maK^m_a(\Omega )}$ and $\| \;
\cdot \; \|_{\maW^{m, \infty}(\Omega )}$ the resulting natural norms
on these spaces. We shall need also the ``broken'' version of these
Babu\v{s}ka-Kondrat'ev spaces for our interface problem. Recall the
subdomains $\Omega_k\subset \Omega$, $1\leq k\leq K$. In analogy with
the smooth case, we then define: $\hat \maK^{m}_{a} (\Omega ) := \{ v:
\Omega \to \RR, \, v \in \maK^m_a(\Omega_k), \ \forall 1 \leq k \leq K
\}$, and $\hat \maW^{m, \infty} (\Omega ) := \{ v: \Omega \to \RR, \,
v\in \maW^{m, \infty}(\Omega_k),\ \forall 1 \leq k \leq K \}.$ If
$\maV$ is empty (that is, if the domain $\Omega$ is smooth and the
interface is also smooth and does not touch the boundary), then we set
$\rho \equiv 1$ and our spaces reduce to the broken Sobolev spaces
$\hat H^m(\Omega)$ and $\hat W^{m, \infty}$ introduced in the previous
section, Equation (\ref{eq.def.broken}). As in the smooth case, the
multiplication and differentiation maps $\hat \maW^{m, \infty} (\Omega
) \times \hat\maK^{m}_{a} (\Omega ) \to \hat\maK^{m}_{a} (\Omega )$
and $\pa_i : \hat\maK^{m}_{a} (\Omega ) \to
\tilde\maK^{m-1}_{a-1}(\Omega )$ are continuous. Let $S \subset \pa
\Omega_k$. Also as in the smooth case, we define the spaces
$\maK^{m+1/2}_{a+1/2}(S)$ as the restrictions to $S$ of the functions
$u \in \maK^{m+1}_{a+1}(\Omega_k )$. These spaces have intrinsic
descriptions \cite{AIN,MN2} similar to the usual
Babu\v{s}ka-Kondrat'ev spaces. Note that no ``hat'' is needed for the
boundary version of the spaces $\hat \maK$. Also
$\maK^{m+1/2}_{a+1/2}(S_1 \cup S_2) = \maK^{m+1/2}_{a+1/2}(S_1) \oplus
\maK^{m+1/2}_{a+1/2}(S_2)$, if $S_1$ and $S_2$ are disjoint.

We need to consider the subset $\maV_s$ of $\maV$ consisting of
Neumann-Neumann corners (\ie corners where two edges endowed with
Neumann boundary conditions meet) and non-smooth points of the
interface, which can be described as $\maV_s := \maV \smallsetminus \{
Q\in \maV, \, Q \in \overline{\pa_D\Omega}\}$. Note that, if a point
$Q$ at the intersection of the interface $\Gamma$ and the boundary
falls on an edge with Neumann boundary conditions, then $Q$ is also
included in $\maV_s$. In order to deal with the singularities arising
at the points in $\maV_s$ (which behave differently than the
singularities at the other points of $\maV$), we also need to augment
our weighted Sobolev spaces with a suitable finite-dimensional
space. Namely, for each point $Q\in\maV_s$, we choose a function
$\chi_Q\in \maC^\infty(\bar\Omega)$ that is constant equal to 1 in a
neighborhood of $Q$. We can choose these functions to have disjoint
supports. Let $W_s$ be the linear span of the functions $\chi_{Q}$ for
any $Q \in \maV_s$. We now define the domains and ranges of our
operators. Assume first that $m \ge 1$.
\begin{eqnarray*}
%\begin{gathered}
&  \maD_{a, m} := (\hat\maK_{a+1}^{m+1}(\Omega ) + W_s) \cap \{u=0
  {\mbox{ on }} \pa_D \Omega \} \cap  \{u^+ - u^{-} = 0 \mbox{ on
    } \Gamma \}  \\
& \maR_{a, m} := \hat \maK_{a-1}^{m-1}(\Omega) \oplus
  \maK_{a-1/2}^{m-1/2}(\pa_N \Omega) \oplus
  \maK_{a-1/2}^{m-1/2}(\Gamma).
%\end{gathered}
\end{eqnarray*}
Let us observe that, by definition, the functions in $W_s$ satisfy the
interface and boundary conditions (so $W_s \subset V :=
  H^1_D(\Omega)$).
%, and hence
% \begin{equation}\label{eq.def.makd}
%  \maD_{a, m} = \Big ( \hat\maK_{a+1}^{m+1}(\Omega ) \cap \{u=0 {\mbox{
%    on }} \pa_D \Omega \} \cap \{u^+ - u^{-} = 0 \mbox{ on } \Gamma \}
%  \Big ) + W_s.
% \end{equation}
Moreover, for $a \ge 0$, we have $\maD_{a, m} =
(\hat\maK_{a+1}^{m+1}(\Omega ) + W_s) \cap H^1_D(\Omega).$ Denote $A =
(a^{ij}, b^{i}, c) \in \hat \maW^{m,\infty}(\Omega )^{d^2 + d + 1}$,
$d=2$, and $P^{A} u = \sum_{i,j=1}^{2} \pa_i(a^{ij} \pa_j u) +
\sum_{i=1}^{2} b^{i} \pa_i u + c u$, as before. Then the family of
partial differential operators $\tilde P^A_{a, m} : \maD_{a, m} \to
\maR_{a, m}$
\begin{equation}
%\begin{gathered}
 %
  \tilde P^A_{a, m} u = \big (Pu, \nabla^A_{\nu} u\vert_{\pa_N
    \Omega},  (\nabla^A_{\nu} u^{+}- \nabla^A_{\nu} u^{+})
    \vert_{\Gamma}  \big)
%\end{gathered} 
\end{equation}
is well defined and the induced map $\hat \maW^{m,\infty}(\Omega
)^{(d^2+d+1)} \ni A = (a^{ij}, b^{i}, c) \to P^{A}_{a, m} \in \maL(
\maD_{a, m},\maR_{a, m})$ is continuous (recall that $d=2$). The
continuity of this map motivates the use of the spaces $\hat
\maW^{m,\infty}(\Omega )$.

When $m=0$, we define 
\begin{eqnarray*}
%\begin{gathered}
 & \maD_{a,m} = \maD_{a,0} = \maK_{a+1}^1(\Omega) \cap \{u=0 {\mbox{ on
  }} \pa_D \Omega \} + W_s\\
 & \maR_{a,m} = \maR_{a,0} = (\maK_{-a+1}^1(\Omega) \cap \{u=0 {\mbox{ on
  }} \pa_D \Omega \})^*,
%\end{gathered}
\end{eqnarray*}
where in the last equation the dual is defined as the dual with pivot
$L^2(\Omega )$. Then we define $\tilde P_{a,0}$ in a weak sense using
the bilinear form $B$ introduced in Equation (\ref{eq:defby}), as in
the smooth case.

Recall the constant $0 < r$ in the definition of the uniform strict
positivity (Definition \ref{def.u.positive}) and Proposition
\ref{prop.existence}. We now state the main result of this
section. Recall that $U$ is reduced to a point in this section.

\begin{theorem} \label{thm.one.op}
Assume that $A = (a^{ij}, b^{i}, c) \in \hat
\maW^{m,\infty}(\Omega)^{d^2 +d+1}$, $d=2$, and that $P^A$ is strictly
positive definite on $V = H^1_{D}(\Omega )$. Then there exists $0 <
\eta $ such that for any $m \in \NN_0$ and for any $0 < a< \eta$, the
map $\tilde P^A_{a,m} : \maD_{a,m} \to \maR_{a,m}$ is boundedly
invertible and $\|(\tilde P^A_{a, m})^{-1}\| \le \tilde C ,$ where
$\tilde C = \tilde C(m,r,a,\|A\|_{\hat \maW^{m, \infty}(\Omega )})$
depends only on the indicated variables.
\end{theorem}

A more typical formulation is given in the following corollary.

\begin{corollary} \label{cor.one.op}
We use the notation and the assumptions of Theorem \ref{thm.one.op}.
If $f \in \hat \maK^{m-1}_{a-1}(\Omega)$, $g \in
\maK^{m-1/2}_{a-1/2}(\pa_N \Omega)$, and $h \in
\maK^{m-1/2}_{a-1/2}(\Gamma)$, then the solution $u \in H^1_D(\Omega)$
of Problem (\ref{eq.pBVP}) can be written $u =u_r+u_s$, with $u_r \in
\hat \maK^{m+1}_{a+1}(\Omega )$ and $u \in W_s$, such that
\begin{eqnarray*}
  \|u_r \|_{\hat\maK^{m+1}_{a+1}} + \|u_s\|_{L^2} \, \le \, \tilde C
  \big ( \|f\|_{\hat \maK^{m-1}_{a-1}} +
  \|g\|_{\maK^{m-1/2}_{a-1/2}(\pa_N \Omega )} +
  \|h\|_{\maK^{m-1/2}_{a-1/2}(\Gamma)} \big ),
\end{eqnarray*}
with $\tilde C$ as in Theorem \ref{thm.one.op}. 
\end{corollary}

\section{Applications\label{sec.fam.reg}}

We keep the settings and notations of the previous section. 
In particular, $d=2$ and we are dealing with equations (not systems).
One can proceed as in \cite{Schwab1,Schwab2,NS1} to obtain
$h^m$-quasi-optimal rates of convergence for the Galerkin $u_n$
approximations of $u$. Namely, under suitable additional regularity in
the $y \in U$ variable one can construct a sequence of finite
dimensional subspaces $S_n \subset L^2(U; V)$ such that
\begin{equation}
  \|u-u_n\|_{L^2(U; V)} \le C \dim(S_n)^{-m/2}\|f\|_{H^{m-1}(\Omega)}.
\end{equation}
This is based on a holomorphic regularity in $U$ and on the
approximation properties in \cite{HMN}. We now state a uniform shift
theorem for our families of boundary value/transmission problems.

Let us denote by $\mcb^{k}(U; Z)$ the space of $k$-times boundedly
differentiable functions defined on $U$ with values in the Banach
space $Z$. By $\mcb^{\omega}(U; Z)$ we shall denote the space of {\em
  analytic} functions with bounded derivatives defined on $U$ with
values in the Banach space $Z$. Recall that $r$ is the constant
appearing in the definition of uniform positivity of the family
\famP. Theorem \ref{thm.one.op} extends to families of boundary value
problems as in \cite{NS1} as follows.  Let us denote by $\eta(y)$ the
best constant appearing in Theorem \ref{thm.one.op} for $P = P_y$ and
$\eta = \inf_{y\in U} \eta(y)$.

\begin{theorem}  \label{thm.fam.op}
Let $m \in \NN_0$ and $k_0 \in \NN_0 \cup \{\infty,\, \omega\}$ be
fixed. Assume that $A = (a^{ij}, b^i, c) \in \mcb^{k_0}(U; \hat
\maW^{m,\infty}(\Omega ))^{d^2+d+1}$, $d=2$, and that the family
$P^A_y$ is uniformly positive definite. Then $\eta = \inf_{y\in U}
\eta(y) > 0$.  Let $f \in \mcb^{k_0}(U; \hat \maK^{m-1}_{a-1}(\Omega
))$, $g \in \mcb^{k_0}(U; \maK^{m-1/2}_{a-1/2}(\pa_N \Omega ))$, $h
\in \mcb^{k_0}(U; \maK^{m-1/2}_{a-1/2}(\Gamma))$, and $0 < a < \eta$.
Then the solution $u$ of our family of boundary value problems
(\ref{eq.pBVP}) satisfies $u \in \mcb^{k_0}(U; \maD_{a,m}))$.
Moreover, for each finite $k \le k_0$, there exists a constant
$C_{a,m} > 0$ such that
\begin{eqnarray*}
  \|u \|_{\mcb^{k}(U; \maD_{a,m})} \, \le \, && C_{a,m} \big(
  \|f\|_{\mcb^{k}(U; \hat \maK^{m-1}_{a-1}(\Omega))} \\
&&  + \|g\|_{\mcb^{k}(U; \maK^{m-1/2}_{a-1/2}(\pa_N \Omega ))} +
  \|h\|_{\mcb^{k}(U; \maK^{m-1/2}_{a-1/2}(\Gamma))} \big) \;.
\end{eqnarray*}
The constant $C_{a,m}$ depends only on $r$, $m$, $a$, $k$, and the
norms of the coefficients $a^{ij}, b^{i}, c$ in $\mcb^{k}(U; \hat
\maW^{m,\infty}(\Omega ))$, but not on $f$ or $g$.
\end{theorem}

\def\cprime{$'$} \def\ocirc#1{\ifmmode\setbox0=\hbox{$#1$}\dimen0=\ht0
  \advance\dimen0 by1pt\rlap{\hbox to\wd0{\hss\raise\dimen0
  \hbox{\hskip.2em$\scriptscriptstyle\circ$}\hss}}#1\else {\accent"17 #1}\fi}

\end{document}